\pdfoutput=1
\documentclass{amsart}
\usepackage{graphicx}
\usepackage{color}
\usepackage{array}
\usepackage[margin=1in, centering]{geometry}
\usepackage{amsmath,amsthm}
\usepackage{amssymb,latexsym}
\usepackage{mathrsfs}
\usepackage{cancel}
\usepackage[initials,short-journals,short-months]{amsrefs}
\usepackage[pdfstartview={FitH}]{hyperref}
\usepackage{url}
\usepackage{enumitem}
\usepackage{soul}
\usepackage{bbm}
\usepackage[FIGTOPCAP]{subfigure}

\newtheorem{theorem}{Theorem}
\newtheorem{lemma}[theorem]{Lemma}
\newtheorem{proposition}[theorem]{Proposition}

\theoremstyle{definition}

\newtheorem{assumption}{Assumption}

\theoremstyle{remark}
\newtheorem{remark}[theorem]{Remark}
\newtheorem*{remark*}{Remark}

\DefineSimpleKey{bib}{myurl}

\newcommand\myurl[1]{\url{#1}}

\BibSpec{webpage}{%
  +{}{\PrintAuthors} {author}
  +{,}{ \textit} {title}
  +{}{ \parenthesize} {date}
  +{,}{ \myurl} {myurl}
  +{,}{ } {note}
  +{.}{ } {transition}
}

\def \CSmooth(#1,#2){\mathcal{C}_{#1,#2}}

\def \Mgale(#1,#2){M_{#1}^{#2}}
\def \Ngale(#1,#2){\mathcal{N}_{#1}^{#2}}

\newcommand{\avg}[2]{{\rm Av}_{#2}\left[{#1}\right]}

\title[Octopus inequality \& moving particle lemma]{The moving particle lemma for the exclusion process on a weighted graph}
\author{Joe P. Chen}
\address{Department of Mathematics, University of Connecticut, Storrs, CT 06269, USA.}
\curraddr{\textsc{Department of Mathematics, Colgate University, Hamilton, NY 13346, USA, \and
Institut Henri Poincar\'e, UMS 839 (CNRS/UPMC), 11 rue Pierre et Marie Curie, 75231 Paris Cedex 05, France.}}
\email{jpchen@colgate.edu}
\urladdr{\url{http://math.colgate.edu/~jpchen}}

\thanks{Research partially supported by NSF grants DMS-1106982 and DMS-1262929.}
\date{\today}

\begin{document}
\renewcommand{\theequation}{\thesection.\arabic{equation}}
\numberwithin{equation}{section}

\begin{abstract}
We prove a version of the moving particle lemma for the exclusion process on any finite weighted graph, based on the octopus inequality of Caputo, Liggett, and Richthammer. In light of their proof of Aldous' spectral gap conjecture, we conjecture that our moving particle lemma is optimal in general. Our result can be applied to graphs which lack translational invariance, including, but not limited to, fractal graphs. An application of our result is the proof of local ergodicity for the exclusion process on a class of weighted graphs, the details of which are reported in a follow-up paper [\href{http://arxiv.org/abs/1705.10290}{arXiv:1705.10290}].
\end{abstract}

\maketitle

\section{Introduction and main result}

The exclusion process is one of the most well-studied interacting particle systems in probability theory; see \cite{AldousFill, Spitzer, IPSStFlour} for introductory accounts of the model, \cite{KipnisLandim, LiggettBook, Spohn} for technical backgrounds, and \cite{ABDS13, BDGJL15} and references therein for connections with non-equilibrium statistical mechanics. In this short paper we consider the exclusion process on a finite weighted graph. To fix notation, let $G=(V,E)$ be a finite connected undirected graph, and ${\bf c}=(c_{xy})_{xy\in E}$ be a collection of nonnegative real numbers called \emph{conductances}. A weighted graph is a pair $(G,{\bf c})$. The symmetric exclusion process (SEP) on $(G,{\bf c})$ is a continuous-time Markov chain on the state space $\{0,1\}^V$ with infinitesimal generator
\begin{align}
\label{LSEP}
\left(\mathcal{L}^{\rm EX}_{(G,{\bf c})} f\right)(\zeta) = \sum_{xy\in E} c_{xy} (\nabla_{xy} f)(\zeta), \quad f: \{0,1\}^V\to\mathbb{R},
\end{align}
where $(\nabla_{xy} f)(\zeta)  = f(\zeta^{xy})-f(\zeta)$ and
\begin{align*}
(\zeta^{xy})(z) = \left\{\begin{array}{ll} \zeta(y), &\text{if}~z=x,\\ \zeta(x), &\text{if}~z=y,\\ \zeta(z), &\text{otherwise}.\end{array}\right. 
\end{align*}
Informally speaking, one starts with a configuration $\zeta$ in which $k$ of the $|V|$ vertices are occupied with a particle, and the remaining vertices are empty. All particles are deemed indistinguishable. A transition from $\zeta$ to $\zeta^{xy}$ occurs with rate $c_{xy}$ if and only if one of the vertices $\{x,y\}$ is occupied and the other is empty.

There are two key properties of the SEP. First, the total number of particles is conserved in the process. Second, the process is reversible with respect to any constant-density product Bernoulli measure $\nu_\alpha$ on $\{0,1\}^V$, $\alpha\in [0,1]$, which has marginal $\nu_\alpha\{\zeta: \zeta(x)=1\}=\alpha
$ for all $x\in V$.
We introduce the Dirichlet energy, \emph{cf.\@} \cite{KipnisLandim}*{p.~343}:
\begin{align}
\mathcal{E}^{\rm EX}_{(G,{\bf c}), \nu_\alpha}(f) = \nu_\alpha\left[f(-\mathcal{L}^{\rm EX}_{(G,{\bf c})} f) \right] = \frac{1}{2}\sum_{xy\in E} c_{xy} \,\nu_\alpha\left[(\nabla_{xy} f)^2\right], \quad f: \{0,1\}^V \to \mathbb{R},
\end{align}
where, throughout the paper, we adopt the shorthand $\mu[h] := \int\, h\,d\mu$ for a Borel measure $\mu$ and a function $h\in L^1(\mu)$. 

Our main result is the following inequality called the \emph{moving particle lemma}.
\begin{theorem}
\label{thm:dirichletex}
For every $\alpha\in [0,1]$, $x,y \in V$, and $f: \{0,1\}^V \to\mathbb{R}$,
\begin{equation}
\label{ineq:dirichletex}
\frac{1}{2} \nu_{\alpha}[(\nabla_{xy} f)^2] \leq R_{\rm eff}(x,y) \mathcal{E}^{\rm EX}_{(G,{\bf c}), \nu_\alpha}(f),
\end{equation}
where $R_{\rm eff}(\cdot,\cdot): V\times V\to \mathbb{R}_+$ is the effective resistance on $(G,{\bf c})$ defined by
\begin{align}
\label{Reff}
[R_{\rm eff}(x,y)]^{-1} &= \inf\left\{ \sum_{zw\in E} c_{zw} [h(z)-h(w)]^2 ~\Bigg|~ h:V\to\mathbb{R},~ h(x)=1,~h(y)=0\right\}\\
\label{Reff2} &\overset{\text{or}}{=} \inf\left\{\frac{\sum_{zw\in E} c_{zw} [h(z)-h(w)]^2}{[h(x)-h(y)]^2} ~\Bigg|~ h:V\to\mathbb{R} \right\},
\end{align}
\emph{cf.\@} \cite{KumagaiStFlour}*{(1.14)} or \cite{StrichartzBook}*{(1.6.1) \& (1.6.2)}.
\end{theorem} 

\begin{remark}
In \eqref{Reff} and \eqref{Reff2} above we defined the effective resistance with respect to a (unit) voltage drop. It is also possible to define the effective resistance with respect to a unit current flow, \emph{cf.\@} \cite{MCMT}*{p.~121}. From the physical point of view, the effective resistance is the power (energy per unit time) dissipated in a unit flow $I$. Writing $I_e$, $\Delta_e V$, and $R_e$ for, respectively, the current, the voltage drop, and the resistance across the edge $e$, we obtain
\[
\sum_e I_e \Delta_e V = \sum_e \frac{[\Delta_e V]^2}{R_e} = \sum_e c_e [\Delta_e V]^2,
\]
where Ohm's law $\Delta_e V = R_e I_e$ was used.
\end{remark}

Theorem \ref{thm:dirichletex} says that the cost of transporting a particle from $x$ to $y$ in the exclusion process is bounded above by the effective resistance distance w.r.t.\@ the \emph{random walk} process times the total energy in the exclusion process. On the one hand, it is reminiscent of the inequality appearing in the classical \emph{Dirichlet's principle}, due to Thomson \cite{ThomsonTait}*{\S376, p.~443} (according to the note in \cite{DoyleSnell}*{Exercise 1.3.11}):
\begin{align}
\label{ineq:dirichletel}
[h(x)-h(y)]^2 \leq R_{\rm eff}(x,y) \mathcal{E}^{\rm el}_{(G,{\bf c})}(h), \quad x,y\in V, \quad h:V\to\mathbb{R},
\end{align}
where
\begin{align}
\label{eq:elenergy}
\mathcal{E}^{\rm el}_{(G,{\bf c})}(h) = \sum_{zw\in E} c_{zw}[h(z)-h(w)]^2
\end{align}
is the Dirichlet energy associated with the symmetric random walk process on $(G, {\bf c})$.   (Observe that \eqref{ineq:dirichletel} follows directly from \eqref{Reff2}. Also note the absence of the prefactor $\frac{1}{2}$ in \eqref{ineq:dirichletel}; the sum in \eqref{eq:elenergy} runs over edges, not vertices.) Notably, the inequality \eqref{ineq:dirichletel} saturates to an equality if and only if $h$ is a harmonic function on $V\setminus \{x,y\}$, \emph{i.e.,} the infimum in the RHS of \eqref{Reff} is a minimum. (Equivalently, and in physics-friendly terms, $R_{\rm eff}(x,y)$ is obtained by minimizing the power dissipation over all unit current flows in $(G, {\bf c})$ from $x$ to $y$, subject to Kirchhoff's current and voltage laws.)

On the other hand, the author does not know the conditions under which \eqref{ineq:dirichletex} saturates to an equality on a general weighted graph. To wit, consider the optimization problem
\begin{align}
\label{eq:opt}
\inf\left\{ \mathcal{J}^{x,y}(f)~|~ f: V\to\mathbb{R} \right\} \quad \text{where} \quad \mathcal{J}^{x,y}(f) = \frac{2\mathcal{E}^{\rm EX}_{(G, {\bf c}),\nu_\alpha}(f)}{\nu_\alpha[(\nabla_{xy} f)^2]}.
\end{align}
Since the functional $\mathcal{J}^{x,y}$ is nonnegative and lower semicontinuous, there exists a minimizer for \eqref{eq:opt}. Then Theorem \ref{thm:dirichletex} says that
\begin{align}
\label{Reffinf}
R_{\rm eff}(x,y) \geq \left(\inf \{\mathcal{J}^{x,y}(f) ~|~ f: V\to\mathbb{R}\}\right)^{-1}.
\end{align}
However, it is unclear in general when \eqref{Reffinf} becomes an equality. 

That being said, we conjecture that Theorem \ref{thm:dirichletex} is optimal on any finite weighted graph, and more generally, on resistance spaces \cite{KigamiRF}, in light of the connection between the moving particle lemma and the spectral gap of an interacting particle system \cite{Morris, Jara, Quastel}. As will be described later in the paper, our proof of Theorem \ref{thm:dirichletex} is based on the ``octopus inequality'' of Caputo, Liggett, and Richthammer \cite{CLR09}, which is a nontrivial energy inequality associated to the interchange process on a weighted graph. (For the definition of the interchange process, see \S\ref{sec:interchange}.)  It is not known in general when the octopus inequality saturates to an equality. Nevertheless, using the octopus inequality, the authors of \cite{CLR09} were able to prove the equality between the spectral gap of the interchange process and the spectral gap of the random walk process, thereby positively resolving Aldous' conjecture (circa 1992 \cite{AldousSGAP}, \emph{cf.\@} \cite{AldousFill}*{Chapter 14, Open Problem 29}). This implies, via a projection argument, that the spectral gap of the exclusion process equals the spectral gap of the random walk process. 

For more information about the effective resistance, see \cite{DoyleSnell} for an elementary exposition, as well as the relevant chapters in \cite{LyonsPeres, MCMT}. Note that $R_{\rm eff}(\cdot, \cdot)$ defines a metric on $V$ (\emph{cf.\@} \cite{MCMT}*{Exercise 9.8} and \cite{KigamiBook}*{Theorem 2.1.14}). 
%For the purpose of this paper, we emphasize the interpretation of the effective resistance distance as an averaged distance over paths.

\subsection{Raison d'\^{e}tre for Theorem \ref{thm:dirichletex}}

To illustrate the difference between our Theorem \ref{thm:dirichletex} and previous results, we quickly recap the argument which leads to the conventional moving particle lemma (\emph{cf.\@} \cite{GPV88}*{Lemma 4.4}, \cite{KOV89}*{pp.\@ 123--124}, \cite{KipnisLandim}*{p.\@ 95}); see \eqref{ineq:convMPL} below. For simplicity assume $(G,{\bf c})$ has all conductances equal to 1. Start by identifying a shortest path connecting $x$ and $y$
\begin{align*}
\{x_0=x, \,x_1,\,\cdots, \,x_{L-1},\, x_L=y\, |\, x_i x_{i+1} \in E\},
\end{align*}
and then swap particle configurations along the edges of the path in this order,
\begin{align*}
 x_0x_1, \,x_1 x_2,\, \cdots, \,x_{L-1} x_L, x_{L-1}x_{L-2},\, x_{L-2} x_{L-3}, \,\cdots, \,x_1 x_0.
\end{align*}
This sends $\zeta(y)$ to $x$, $\zeta(x)$ to $y$, and leaves $\zeta(z)$ intact for all $z \notin \{x,y\}$. (In a related context, \cite{DSC93} uses this path argument and a comparison argument to obtain eigenvalue estimates in the $k$-particle exclusion process on $(G,{\bf c})$.) Let us denote each edge-wise swap by the operator $D_m: \{0,1\}^V\to\{0,1\}^V$, $m\in \{1,2,\cdots, 2L-1\}$, in the order shown, set $T_1={\rm Id}$ and, for $m\geq 2$, $T_m = D_{m-1} D_{m-2}\cdots D_1$. (In particular, $D_L$ represents the swap between particles at $x_L$ and $x_{L-1}$, while $D_{L+1}$ the swap between particles at $x_{L-1}$ and $x_{L-2}$.) Now use the telescoping identity
\begin{align*}
(\nabla_{xy} f)(\zeta) = \sum_{m=1}^L (\nabla_{x_{m-1} x_m} f)(T_m \zeta) + \sum_{k=1}^{L-1} (\nabla_{x_{L-k} x_{L-k-1}}f)(T_{L+k}\zeta),
\end{align*}
then apply the Cauchy-Schwarz inequality:
\begin{align}
\label{ineq:path}
[(\nabla_{xy} f)(\zeta)]^2 \leq (2L-1)\left[\sum_{m=1}^L [(\nabla_{x_{m-1} x_m} f)(T_m \zeta)]^2 + \sum_{k=1}^{L-1} [(\nabla_{x_{L-k} x_{L-k-1}}f)(T_{L+k}\zeta)]^2\right].
\end{align}
Integrating both sides of \eqref{ineq:path} w.r.t.\@ the uniform probability measure $\nu$ on $\{0,1\}^V$, and noting the transposition invariance of the measure $\nu$, we obtain 
\begin{align}
\label{preconvMPL}
\nu[(\nabla_{xy} f)^2] \leq (2L-1) \left(\sum_{m=1}^L \nu[(\nabla_{x_{m-1} x_m} f)^2] + \sum_{k=1}^{L-1} \nu[(\nabla_{x_{k-1} x_k} f)^2]\right) \leq 2(2L-1) \sum_{m=1}^L \nu[(\nabla_{x_{m-1} x_m} f)^2] .
\end{align}

Besides \eqref{preconvMPL}, one also needs to verify that the energy distribution over each edge is uniformly bounded from above.
\begin{assumption}
\label{ass:A}
There exists a positive constant $C$, possibly depending on $f$, such that 
\begin{align*}
\frac{\frac{1}{2}\nu[(\nabla_e f)^2] \cdot |E|}{\mathcal{E}^{\rm EX}_{(G,{\bf c}),\nu}(f)} \leq C
\end{align*}
for all $e\in E$.
\end{assumption}

Combining Assumption \ref{ass:A} with \eqref{preconvMPL} we get
\begin{align}
\label{ineq:convMPL}
\frac{1}{2}\nu[(\nabla_{xy} f)^2] \leq 2 (2L-1) L \cdot  \frac{C}{|E|} \mathcal{E}^{\rm EX}_{(G,{\bf c}),\nu}(f) \leq C \frac{4L^2}{|E|} \mathcal{E}^{\rm EX}_{(G,{\bf c}),\nu}(f).
\end{align}

Inequality \eqref{ineq:convMPL} is the conventional moving particle lemma. It is effective when the graph $(G,{\bf c})$ is quasi-transitive w.r.t.\@ a free group action (such as the group of translations), and $f$ is taken to be invariant under the group action. The canonical example is the $d$-dimensional discrete torus $\mathbb{T}^d_N:=(\mathbb{Z}/N\mathbb{Z})^d$: if $f: \mathbb{T}^d_N\to\mathbb{R}$ is invariant under lattice translations and rotations by $\pi/2$, then Assumption \ref{ass:A} holds with equality and $C=1$, and \eqref{ineq:convMPL} becomes
\begin{align}
\frac{1}{2}\nu[(\nabla_{xy} f)^2] \leq \frac{4\|x-y\|_1^2}{N^d} \mathcal{E}^{\rm EX}_{\mathbb{T}^d_N,\nu}(f), \quad x,y\in \mathbb{T}^d_N,
\end{align}
where $\|x\|_1 := \sum_{i=1}^d |x_i|$ is the $L^1$-distance on $\mathbb{T}^d_N$ \cite{GPV88, KOV89, KipnisLandim}. A closely related example is a crystal lattice, which is analyzed in \cite{TanakaCrystal}. (See \cite{TanakaCrystal}*{Lemma 4.2} for the moving particle lemma there. It has since been extended to finitely generated residually finite amenable groups in \cite{TanakaTower}). We also mention examples of random graphs for which a variation of \eqref{ineq:convMPL} holds almost surely w.r.t.\@ the law of the random environment; see \cite{Quastel}*{Lemma 5.2} for the case of the exclusion process with site disorder, and \cite{Faggionato}*{Lemma 5.2} for the case of the zero-range process on the supercritical percolation cluster on $\mathbb{Z}^d$. 

However, it may be the case that for certain classes of weighted graphs, one cannot verify Assumption \ref{ass:A} for \emph{any} $f$, which would cast \eqref{ineq:convMPL} in doubt. Consider graphs associated with self-similar fractal sets, such as the Sierpinski gasket ($SG$, see Figure \ref{fig:SGImage}) and the Sierpinski carpet. On these spaces, it is proved that for every $f$ in the domain of the Dirichlet energy for single-particle diffusion, the energy measure is mutually \emph{singular} w.r.t.\@ the Hausdorff measure \cite{Kusuoka, BST99, Hino}. It would then seem plausible that the exclusion process energy measure is also mutually singular w.r.t.\@ the Hausdorff measure, though we leave this as an open problem. At any rate, since we are unable to verify Assumption \ref{ass:A} on these spaces, we use Theorem \ref{thm:dirichletex} instead to capture the averaging property through the effective resistance distance.

\begin{figure}
\centering
\includegraphics[width=0.8\textwidth]{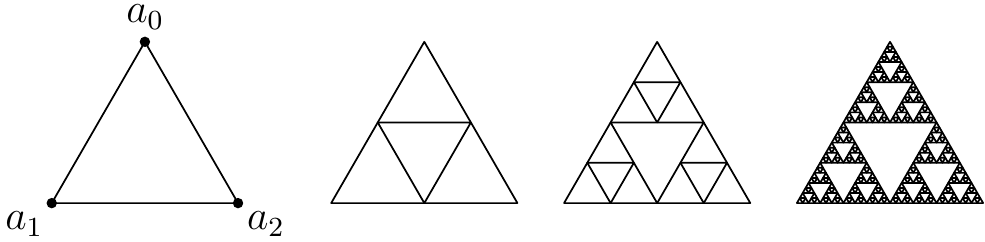}
\caption{The Sierpinski gasket ($SG$) graphs of level $0$, $1$, $2$, and $5$, respectively.}
\label{fig:SGImage}
\end{figure}

\subsection{Application to local ergodicity} \label{sec:fractals}

A preliminary motivation of this paper is to establish a local ergodic theorem for the exclusion process on non-translationally-invariant weighted graphs, such as $SG$ (see Theorem \ref{thm:localrep} below).
Theorem \ref{thm:dirichletex} will enable us to prove the so-called two-blocks estimate and, in turn, local ergodicity, for the said process. The details are reported in an upcoming paper \cite{ChenLocalErgodic}. Then in \cite{ChenTeplyaevSGHydro} we specialize to $SG$, and prove the hydrodynamic limit of the (boundary-driven) exclusion process, \emph{viz.\@} the joint current-density law of large numbers and large deviations principle.  A long-term goal of ours is to establish the hydrodynamic limit of interacting particle systems on the so-called \emph{strongly recurrent} graphs, including post-critically finite self-similar (p.c.f.s.s.\@) fractals \cite{BarlowStFlour, KigamiBook} and Sierpinski carpets \cite{BB99, KusuokaZhou, BBKT10}, whereupon random walks satisfy \emph{sub-Gaussian} heat kernel estimates \cite{BCK05}.

To give a flavor of how Theorem \ref{thm:dirichletex} is applied, we now state the moving particle lemma on $SG$ (which has not appeared in the previous literature according to the author's knowledge), as well as the local ergodic theorem. For discussions of other weighted graphs see \cite{ChenLocalErgodic}. 

Let $a_0$, $a_1$, $a_2$ be the vertices of a nondegenerate triangle in $\mathbb{R}^2$, and $G_0$ be the complete graph on the vertex set $V_0 = \{a_0, a_1, a_2\}$, as shown on the left in Figure \ref{fig:SGImage}. We declare $V_0$ to be the (analytical but not topological) boundary of $SG$. Define the contracting similitude $\Psi_i : \mathbb{R}^2 \to\mathbb{R}^2$, $\Psi_i(x) = \frac{1}{2}(x-a_i) + a_i $ for each $i\in \{0,1,2\}$. For each $N\in \mathbb{N}$, the $N$th-level $SG$ graph $G_N=(V_N, E_N)$ is constructed inductively via the formula $G_N = \bigcup_{i=0}^2 \Psi_i(G_{N-1})$. Set $V_* = \bigcup_{N=0}^\infty V_N$. Finally, set the conductance on every $e\in E_N$ to $1$. We denote the corresponding weighted graph $(G_N, {\bf 1})$.

For each $m$-letter word $w=w_1 w_2 \cdots w_m \in \{0,1,2\}^m$ , put $\Psi_w = \Psi_{w_1} \circ \Psi_{w_2} \circ \cdots \circ \Psi_{w_m}$. Two vertices $x, y\in V_N$ are said to be in the same level-$j$ cell, $j\in \{0,1,\cdots, N\}$, if $x, y\in \Psi_w(V_0)$ for some $j$-letter word $w \in \{0,1,2\}^j$.

\begin{proposition}[Moving particle lemma on $SG$]
\label{prop:SGMPL}
There exists a positive constant $C$, independent of $N$, such that for every $\alpha \in [0,1]$, $x,y \in V_N$ which are in the same level-$j$ cell, and $f: V_N\to\mathbb{R}$, we have
\begin{align}
\label{SGMPL}
\nu_\alpha\left[(\nabla_{xy} f)^2\right] \leq C \left(\frac{5}{3}\right)^{N-j} \mathcal{E}^{\rm EX}_{(G_N,{\bf 1}), \nu_\alpha}(f).
\end{align}
\end{proposition}
 
\begin{proof}
It is known (see \emph{e.g.\@} \cite{StrichartzBook}*{Lemma 1.6.1}) that in the graph $(G_N,{\bf 1})$, for every $x$ and $y$ in the same level-$j$ cell, $R_{\rm eff}(x,y) \leq C \left(\frac{5}{3}\right)^{N-j}$ for a constant $C$ independent of $N$ and $j$. Inequality \eqref{SGMPL} follows directly from Theorem \ref{thm:dirichletex} and this resistance estimate. 
\end{proof}

Note that the value $\frac{5}{3}$ is the single-particle diffusive scaling on $SG$, which is crucial to the proof of the two-blocks estimate. Indeed, it suffices to consider all $j=\lfloor \epsilon N\rfloor$ with $\epsilon \in [0,1]$ ($\epsilon$ is the macroscopic aspect ratio used in the coarse-graining argument), and all functions $f$ with $\mathcal{E}^{\rm EX}_{(G_N, {\bf 1}),\nu_\alpha}(f) \leq C \left(\frac{3}{5}\right)^N$. Then the cost of particle transport, \emph{viz.\@} the LHS of \eqref{SGMPL}, is at most of order $\left(\frac{3}{5}\right)^{\lfloor \epsilon N\rfloor}$, which vanishes in the double limit $N\to\infty$ followed by $\epsilon\downarrow 0$. This asymptotic statement is crucial for the two-blocks estimate to go through.

Once the one-block and two-blocks estimates are proved, we can then prove the local ergodic theorem for the exclusion process on $SG$, see Theorem \ref{thm:localrep} below. (In \cite{ChenLocalErgodic} we shall state and prove a more abstract version of the local ergodic theorem which applies to all \emph{strongly recurrent} weighted graphs, in the sense of \cite{BCK05}, which includes $SG$.) Given a denumerable set $\Lambda$, let $|\Lambda|$ denote the cardinality of $\Lambda$. The average of $g: \Lambda\to\mathbb{R}$ over $\Lambda$ is written $\avg{g}{\Lambda} :=|\Lambda|^{-1} \sum_{z\in \Lambda} g(z)$. Let $B_d(x,r) = \{y\in V_*: d(x,y)<r\}$ denote the ball of radius $r$ in the graph metric $d$ centered at $x$. A map $\phi: V_* \times \{0,1\}^{V_*} \to\mathbb{R}$ is called a \emph{local function bundle for vertices} (this terminology comes from \cite{TanakaTower}) if there exists $r_\phi \in (0,\infty)$ such that for any $x\in V_*$, $\phi_x:=\phi(x,\cdot)$ depends only on $\{\eta(z): z\in B_d(x,r_\phi)\}$.

\begin{theorem}[Local ergodic theorem for the exclusion process on $SG$]
\label{thm:localrep}
Let $\mathbb{P}^N_{\alpha}$ be the law of the symmetric exclusion process $(\eta_t^N)_{t\geq 0}$ with generator $5^N \mathcal{L}^{\rm EX}_{(G_N, {\bf 1})}$, started from the product Bernoulli measure $\nu_\alpha$ on $\{0,1\}^{V_N}$ with marginals $\nu_\alpha\{\eta: \eta(x)=1\} =\alpha$ for all $x\in V_N$. Then for each $T>0$ and each $\delta>0$,
\begin{equation}
\label{supexp}
\limsup_{\epsilon\downarrow 0} \limsup_{N\to\infty} \sup_{x\in V_N} \frac{1}{3^N} \log \mathbb{P}^N_{\alpha} \left\{ \left|\int_0^T\, U_{N,\epsilon}(x,\eta^N_t)\,dt\right|>\delta\right\} = -\infty,
\end{equation}
where
\begin{align*}
U_{N,\epsilon}(x,\eta) &=  \phi_x(\eta) -  \Phi_x\left(\avg{\eta}{B_d(x,2^{\lfloor \epsilon N\rfloor})}\right), \quad \Phi_x(\alpha) := \nu_\alpha[\phi_x],
\end{align*}
and $\phi$ is any local function bundle for vertices.

\end{theorem}

\begin{remark}
From the point of view of non-equilibrium statistical mechanics, one may also consider the boundary-driven version of the exclusion process on $(G,{\bf c})$, following \cite{BDGJL03, BDGJL07}. In this setting, we declare a nonempty subset $\partial V \subset V$ to be the boundary set, and assume for simplicity that $c_{aa'}=0$ for all $a,a' \in \partial V$. Attach to each $a\in \partial V$ a particle reservoir which imposes a fixed particle density at $a$, resulting in a mean density profile which may be spatially non-constant. 

Formally, the generator of the boundary-driven exclusion process is $\mathcal{L} := \mathcal{L}^{\rm EX}_{(G,{\bf c})} + \mathcal{L}^{b}_{\partial V}$, where
\begin{align}
(\mathcal{L}^b_{\partial V} f)(\zeta)  = \sum_{a\in \partial V} [\lambda_-(a) \zeta(a) + \lambda_+(a) (1-\zeta(a))] [ f(\zeta^a)-f(\zeta)], \quad f: \{0,1\}^V \to\mathbb{R}.
\end{align}
Here $\lambda_+(a) \in \mathbb{R}_+$ (resp.\@ $\lambda_-(a) \in \mathbb{R}_+$) represents the rate of particle hopping into (resp.\@ out of) the reservoir at $a$, and
\begin{align}
\zeta^a(z) = \left\{\begin{array}{ll} 1-\zeta(a),& \text{if}~z=a,\\ \zeta(z), & \text{otherwise.}\end{array}\right.
\end{align} 

In \cite{ChenLocalErgodic} we prove the moving particle lemma for the boundary-driven exclusion process, using Theorem \ref{thm:dirichletex} and potential theoretic estimates in the random walk process. 
\end{remark}

We end this section by mentioning that similar questions can be posed for the \emph{zero-range process} (see \cite{KipnisLandim}*{\S2.3} for the definition of the model) on a finite connected weighted graph. To the best of the author's knowledge, there is no analog of the octopus inequality for the zero-range process, and it remains a challenge to obtain sharp asymptotics of the spectral gap on general graphs. Jara has investigated this question on $SG$ \cite{Jara}*{p.~787} and posed an open problem (that the spectral gap has a uniform bound of order $5^{-N}$).

%Organization
The rest of the paper is organized as follows. In \S\ref{sec:reduction} we briefly recap the idea of electric network reduction, and fix notation which will be used later on. In \S\ref{sec:interchange} we define the interchange process, and show that the octopus inequality implies the monotone decreasing property of the corresponding Dirichlet energy upon network reduction. This leads to a counterpart of Dirichlet's principle for the interchange process. Then in \S\ref{sec:exclusion} we project the interchange process onto a $k$-particle exclusion process, and in conjunction with known properties of the exclusion process, we obtain the desired moving particle lemma, Theorem \ref{thm:dirichletex}.

\section{Electric network reduction} \label{sec:reduction}

In this section we define network reduction, following the notation of \cite{CLR09}*{\S2}. Given a finite weighted graph $(G=(V,E), {\bf c})$ and a vertex $x \in V$, define $V_x = V\setminus \{x\}$, $E_x=\{yz\in E: y,z\neq x\}$, and
\begin{align}
\label{eq:condinc}
\tilde{c}_{yz} := c_{yz} + c^{*,x}_{yz}, \quad c^{*,x}_{yz} := \frac{c_{yx}c_{xz}}{\sum_{w\in V_x} c_{xw}}, \quad yz\in E_x.
\end{align}
We call the weighted graph $\left(G_x=(V_x, E_x), \tilde{\bf c}=(\tilde{c}_{yz})_{yz\in E_x}\right)$ the \emph{reduced (star) graph of $(G,{\bf c})$ at $x$}. In simple terms, $G_x$ is obtained by removing $x$ and its attached edges from $G$. In order to leave the effective conductance between any pair of points invariant, the conductance on each remaining edge $yz\in E_x$ must increase from $c_{yz}$ to $\tilde{c}_{yz}$, or by an amount $c^{*,x}_{yz}$. Formally,  (\ref{eq:condinc}) is obtained by computing the Schur complement of the $(x,x)$ block in the stochastic matrix associated to symmetric random walk on $(G,{\bf c})$. It is direct to verify, via the next proposition, that the effective conductance $c_{\rm eff}(\cdot,\cdot) = [R_{\rm eff}(\cdot,\cdot)]^{-1}$ is invariant under network reduction.

\begin{proposition}[\cite{CLR09}*{Lemma 2.2}]
\label{prop:energystar}
For every $x\in V$ and $f: V\to \mathbb{R}$,
\begin{align}
\label{eq:energyeq}
\sum_{y\in V_x} c_{xy} [f(x)-f(y)]^2 = \sum_{yz\in E_x} c_{yz}^{*,x} [f(y)-f(z)]^2 + \frac{1}{\sum_{y\neq x} c_{xy}}[(Lf)(x)]^2,
\end{align}
where $(Lf)(x) = \sum_{y\in V_x} c_{xy} [f(y)-f(x)]$. It follows that
\begin{align}
\label{eq:energyineq}
\sum_{y\in V_x} c_{xy} [f(x)-f(y)]^2 \geq \sum_{yz\in E_x} c_{yz}^{*,x} [f(y)-f(z)]^2,
\end{align}
with equality holding if and only if $(Lf)(x)=0$.
\end{proposition}

The proof of \eqref{eq:energyeq} is a straightforward algebraic exercise.

We now make the connection between \eqref{eq:energyineq} and the inequality appearing in Dirichlet's principle \eqref{ineq:dirichletel}. Inequality (\ref{eq:energyineq}) says that by fixing a voltage function $f$ and implementing a network reduction, the energy lost due to the the removed edges (LHS of the inequality) is at least the energy gained from the increased conductances on the remaining edges (RHS of the inequality). Of course this is equivalent to saying that the energy is monotone decreasing upon network reduction. Indeed, by adding $\sum_{yz\in E_x} c_{yz}[f(y)-f(z)]^2$ to both sides of (\ref{eq:energyineq}), we get
\begin{align}
\label{energyineq2}
\mathcal{E}_{(G,{\bf c})}^{\rm el}(f) \geq \mathcal{E}_{(G_x, \tilde{\bf c})}^{\rm el}(f).
\end{align}

Let us fix a pair of vertices $x,y \in V$ in the finite weighted graph $(G,{\bf c})$, and label the remaining vertices by $x_1, x_2, \cdots, x_{|V|-2}$. We define a decreasing sequence of weighted graphs $\{(G_i=(V_i, E_i), {\bf c}_i)\}_{i=0}^{|V|-2}$ inductively as follows: Put $(G_0, {\bf c}_0)=(G,{\bf c})$, and for every $0\leq i\leq |V|-3$, let $(G_{i+1}, {\bf c}_{i+1}) = \left((G_i)_{x_{i+1}}, \tilde{\bf c}_i\right)$, where $(G_i)_{x_{i+1}}$ is the subgraph of $G_i$ obtained by removing $x_{i+1}$ and its attached edges, and $(\tilde{c}_i)_{yz} = (c_i)_{yz} + (c_i^{*,x_{i+1}})_{yz}$ for all $yz \in E_{i+1}$, \emph{cf.} (\ref{eq:condinc}). 

Applying the inequality (\ref{energyineq2}) to the sequence of network reductions, we find
\begin{align}
\mathcal{E}_{(G,{\bf c})}^{\rm el}(f) \geq \mathcal{E}_{\left(G_{1}, {\bf c}_{1}\right)}^{\rm el}(f) \geq \cdots \geq \mathcal{E}_{\left(G_{|V|-2}, {\bf c}_{|V|-2}\right)}^{\rm el}(f) = \left({\bf c}_{|V|-2}\right)_{xy} [f(x)-f(y)]^2.
\end{align}
Recognize that $\left({\bf c}_{|V|-2}\right)_{xy}=c_{\rm eff}(x,y) = [R_{\rm eff}(x,y)]^{-1}$. Thus we have proved \eqref{ineq:dirichletel}.

\begin{remark*}
One can show that the condition for equality in \eqref{ineq:dirichletel} follows from the condition for equality in (\ref{energyineq2}). However it is not of central interest to the rest of the paper, so we omit the proof.
\end{remark*}

\section{Moving particle lemma for the interchange process} \label{sec:interchange}

In this section we consider the interchange process on $(G, {\bf c})$. This Markov chain is described informally as follows. Each state corresponds to an assignment of $|V|$ labelled particles to the vertices of $G$ such that each vertex has exactly 1 particle. A transition occurs when particles at vertices $x$ and $y$ interchange their positions at rate $c_{xy}$. Formally, let $S^{\rm IP}$ be the space of permutations on $\{1,2,\cdots, |V|\}$, and for $\eta \in S^{\rm IP}$ and $xy \in E$, let $\eta^{xy}=\eta \tau_{xy}$, where $\tau_{xy} \in S^{\rm IP}$ is the transposition of $x$ and $y$. The generator for the interchange process is
\begin{equation}
(\mathcal{L}^{\rm IP}_{(G,{\bf c})} f)(\eta) = \sum_{xy\in E} c_{xy} [f(\eta^{xy})-f(\eta)], \quad f: S^{\rm IP}\to\mathbb{R}.
 \end{equation}
Henceforth we denote $(\nabla_{xy}f)(\eta) := f(\eta^{xy})-f(\eta)$. Also let $\nu$ be the uniform probability measure on $S^{\rm IP}$, which is the unique reversible invariant measure for this process. The corresponding Dirichlet energy reads
\begin{equation}
\mathcal{E}^{\rm IP}_{(G,{\bf c})}(f) = \nu\left[f (-\mathcal{L}^{\rm IP}_{(G,{\bf c})} f) \right] =\frac{1}{2}\sum_{xy\in E} c_{xy} \,\nu[(\nabla_{xy} f)^2],  \quad f: S^{\rm IP} \to \mathbb{R}.
\end{equation}

The next result, called the \emph{octopus inequality}, is the counterpart of Proposition \ref{prop:energystar} for the interchange process. See \cite{CLR09}*{\S3} for the proof (and also \cite{Cesi} for an algebraic perspective), which involves a series of nontrivial, clever exercise in linear algebra.

\begin{proposition}[Octopus inequality \cite{CLR09}*{Theorem 2.3}]
\label{thm:octopus}
For every $x\in V$ and $f: S^{\rm IP} \to \mathbb{R}$,
\begin{equation}
\label{ineq:octopus}
\sum_{y\in V_x} c_{xy} \,\nu[(\nabla_{yz} f)^2] \geq \sum_{yz \in E_x} c_{yz}^{*,x}\, \nu[(\nabla_{yz} f)^2].
\end{equation}
\end{proposition}

\begin{remark*}
It is unclear to the author under which conditions the inequality (\ref{ineq:octopus}) saturates to an equality on a general weighted graph.
\end{remark*}

Recalling the strategy used in the previous section, we now prove the analog of \eqref{ineq:dirichletel} for the interchange process. This will be referred to as the moving particle lemma for the interchange process. Inequality \eqref{EIPineq2} below will be used in the proof of our main Theorem \ref{thm:dirichletex}. 

\begin{lemma}
\label{thm:dirichletint}
For any $x,y \in V$ and $f: S^{\rm IP} \to\mathbb{R}$,
\begin{align}
\label{eq:mplinter}
\frac{1}{2} \nu[(\nabla_{xy} f)^2]  \leq R_{\rm eff}(x,y)\mathcal{E}^{\rm IP}_{(G,{\bf c})}(f).
\end{align}
\end{lemma}

\begin{proof}
Upon adding $\sum_{yz\in E_x} c_{yz}\,\nu\left[(\nabla_{yz} f)^2\right]$ to both sides of (\ref{ineq:octopus}) and multiplying by $\frac{1}{2}$, we obtain
\begin{align}
\label{EIPineq}
\mathcal{E}^{\rm IP}_{(G,{\bf c})}(f) \geq \mathcal{E}^{\rm IP}_{(G_x, \tilde{\bf c})}(f).
\end{align}
This shows that in the interchange process the Dirichlet energy is monotone decreasing under network reduction. Adopting the same notation as in the paragraph after (\ref{energyineq2}), we apply (\ref{EIPineq}) to the sequence of network reductions $\{(G_i, {\bf c}_i)\}_{i=0}^{|V|-2}$ to get
\begin{align}
\label{EIPineq2}
\mathcal{E}^{\rm IP}_{(G,{\bf c})}(f) \geq \mathcal{E}^{\rm IP}_{(G_1,{\bf c}_1)}(f) \geq \cdots \geq \mathcal{E}^{\rm IP}_{(G_{|V|-2}, {\bf c}_{|V|-2})}(f) =  \frac{1}{2} c_{\rm eff}(x,y) \,\nu[(\nabla_{xy} f)^2].
\end{align}
This proves the lemma.
\end{proof}

\section{Moving particle lemma for the exclusion process} \label{sec:exclusion}

Finally we turn to the main problem considered in this paper. Let $k\in \{0,1,\cdots |V|\}$. The symmetric exclusion process (SEP) of $k$ particles on a finite weighted graph $(G,{\bf c})$ is a Markov chain on the state space $S^{\rm EX}_k = \{\zeta \subset V: |\zeta|=k\}$ generated by
\begin{align}
\label{LEXk}
(\mathcal{L}^{\rm EX}_{(G,{\bf c}),k} f)(\zeta) = \sum_{xy\in E} c_{xy} [f(\zeta^{xy})-f(\zeta)], \quad f: S^{\rm EX}_k\to\mathbb{R},
\end{align}
where
\begin{align}
\zeta^{xy} = \left\{\begin{array}{ll} (\zeta \setminus \{x\}) \cup \{y\},& \text{if}~x\in \zeta ~\text{and}~y \notin \zeta, \\ (\zeta\setminus \{y\}) \cup\{x\} ,&\text{if}~y\in \zeta ~\text{and}~x\notin\zeta, \\ \zeta,& \text{otherwise}.\end{array}\right.
\end{align}

This process can be viewed as a sub-process of the interchange process as follows \cite{CLR09}*{\S4.1.1}. Let $\xi_i(\eta)$ denote the position of particle $i$ in the configuration $\eta\in S^{\rm IP}$. Define the contraction $\pi_k: S^{\rm IP} \to S^{\rm EX}_k$ by
\begin{equation}
\pi_k(\eta) = \{\xi_1(\eta), \cdots, \xi_k(\eta)\}.
\end{equation}
It is direct to verify that $\pi_k(\eta^{xy}) = \left(\pi_k(\eta)\right)^{xy}$. 
Therefore for all $f: S^{\rm EX}_k\to\mathbb{R}$ and $\eta\in S^{\rm IP}$,
\begin{align}
\left(\mathcal{L}^{\rm IP}_{(G,{\bf c})} (f\circ \pi_k)\right)(\eta) &= \sum_{xy\in E} c_{xy} [f(\pi_k(\eta^{xy}) )-f (\pi_k(\eta))]\\
\nonumber &= \sum_{xy\in E} c_{xy} [f((\pi_k(\eta))^{xy})-f(\pi_k(\eta))]= \left((\mathcal{L}^{\rm EX}_{(G,{\bf c}),k} f)\circ \pi_k \right)(\eta).
\end{align}
So if $\nu$ is the uniform probability measure on $S^{\rm IP}$, then
\begin{align}
\label{chofvar}
\mathcal{E}^{\rm IP}_{(G,{\bf c})}(f\circ \pi_k) &=\nu\left[(f\circ \pi_k) (-\mathcal{L}^{\rm IP}_{(G,{\bf c})}(f\circ \pi_k))\right]
=  \nu\left[(f\circ \pi_k)((-\mathcal{L}^{\rm EX}_{(G,{\bf c})} f)\circ \pi_k)\right]\\
\nonumber &=  \int_{S^{\rm IP}} \, (f\circ \pi_k)(\eta) ((-\mathcal{L}^{\rm EX}_{(G,{\bf c})} f)\circ \pi_k)(\eta) \,\nu(d\eta)\\
\nonumber &=  \int_{S^{\rm EX}_k} \, f(\zeta) (-\mathcal{L}^{\rm EX}_{(G,{\bf c})} f)(\zeta)\, (\nu\circ \pi_k^{-1})(d\zeta),
\end{align}
assuming that the integral is finite. A moment's thought tells us that $\nu_k:=\nu\circ \pi_k^{-1}$ is the uniform probability measure on $S^{\rm EX}_k$ (charging $[{n\choose k}]^{-1}$ to each $\zeta \in S^{\rm EX}_k$). Let us denote the Dirichlet energy of the $k$-particle exclusion process w.r.t.\@ $\nu_k$ by
\begin{align}
\label{eEX}
\mathcal{E}^{\rm EX}_{(G,{\bf c}),k}(f) :=  \nu_k\left[f(-\mathcal{L}^{\rm EX}_{(G, {\bf c})}f) \right] =\frac{1}{2}\sum_{xy\in E} c_{xy} \nu_k[(\nabla_{xy}f)^2].
\end{align}
Combining \eqref{chofvar} and \eqref{eEX} yields the identity
\begin{align}
\label{eq:Eproj}
\mathcal{E}^{\rm IP}_{(G,{\bf c})}(f\circ \pi_k) = \mathcal{E}^{\rm EX}_{(G,{\bf c}),k}(f), \quad f: S^{\rm EX}_k \to\mathbb{R}.
\end{align}

If the total number of particles $k$ is unspecified, then the SEP on $(G, {\bf c})$ can be viewed as a Markov chain on the configuration space $S^{\rm EX}=\{0,1\}^V$ with generator
\begin{equation}
(\mathcal{L}^{\rm EX}_{(G,{\bf c})} f)(\zeta) = \sum_{xy\in E} c_{xy}( \nabla_{xy}f)(\zeta), \quad f: S^{\rm EX}\to \mathbb{R},
\end{equation}
where $\nabla_{xy}$ is as defined just after \eqref{LSEP}. Recall that the total particle number is conserved in the SEP. Therefore the generator $\mathcal{L}^{\rm EX}_{(G,{\bf c})}$ admits the orthogonal decomposition $\mathcal{L}^{\rm EX}_{(G,{\bf c})} = \bigoplus_{k=0}^{|V|} \mathcal{L}^{\rm EX}_{(G,{\bf c}),k}$, where each $\mathcal{L}^{\rm EX}_{(G,{\bf c}),k}$, as defined in \eqref{LEXk}, acts on the invariant subspace $S^{\rm EX}_k =\{\zeta \in \{0,1\}^V: \sum_{x\in V} \zeta(x)=k \}$. 

For $\alpha \in [0,1]$, let $\nu_\alpha$ be the product Bernoulli measure on $S^{\rm EX}$ with marginal $\nu_\alpha(\{\zeta:\zeta(x)=1\})=\alpha$ for all $x\in V$. Define the Dirichlet energy w.r.t.\@ $\nu_\alpha$ by
\begin{equation}
\mathcal{E}^{\rm EX}_{(G,{\bf c}),\nu_\alpha}(f) = \nu_\alpha\left[f (-\mathcal{L}^{\rm EX}_{(G,{\bf c})} f) \right]= \frac{1}{2}\sum_{xy\in E} c_{xy} \nu_\alpha[(\nabla_{xy} f)^2], \quad f: S^{\rm EX}\to \mathbb{R}.
\end{equation}

\begin{lemma}
\label{lem:chambers}
For every $f: S^{\rm EX}\to\mathbb{R}$,
\begin{equation}
\label{eq:Eex}
\mathcal{E}^{\rm EX}_{(G,{\bf c}),\nu_\alpha}(f) = \sum_{k=0}^{|V|} {|V| \choose k}\alpha^k (1-\alpha)^{|V|-k}\mathcal{E}^{\rm EX}_{(G,{\bf c}),k}(f_k) =\sum_{k=0}^{|V|} {|V|\choose k}\alpha^k (1-\alpha)^{|V|-k}\mathcal{E}^{\rm IP}_{(G,{\bf c})}(f_k\circ \pi_k)  ,
\end{equation}
where $f_k$ denotes the orthogonal projection of $f$ onto $S^{\rm EX}_k$.
\end{lemma}
\begin{proof}
On the one hand, using the orthogonal decomposition of $\mathcal{L}^{\rm EX}_{(G,{\bf c})}$, we find
\begin{align}
\mathcal{E}^{\rm EX}_{(G,{\bf c}), \nu_\alpha}(f) = \nu_\alpha\left[ f(-\mathcal{L}^{\rm EX}_{(G,{\bf c})} f)\right] =\sum_{k=0}^{|V|}\nu_\alpha\left[ f_k (-\mathcal{L}^{\rm EX}_{(G,{\bf c}),k} f_k) \right].
\end{align}
On the other hand, the restriction of $\nu_\alpha$ to $S^{\rm EX}_k$ is a uniform measure with total mass ${|V|\choose k}\alpha^k (1-\alpha)^{|V|-k}$. Referring to \eqref{eEX} we see that 
\begin{align}
\nu_\alpha\left[ f_k (-\mathcal{L}^{\rm EX}_{(G,{\bf c}),k} f_k) \right] = {|V|\choose k}\alpha^k (1-\alpha)^{|V|-k}\mathcal{E}^{\rm EX}_{(G,{\bf c}),k}(f_k).
\end{align}
These two observations together justify the first equality in (\ref{eq:Eex}). The second equality in \eqref{eq:Eex} follows from (\ref{eq:Eproj}).
\end{proof}

We are now in a position to prove our main result.

\begin{proof}[Proof of Theorem \ref{thm:dirichletex}]
We carry out the sequence of network reductions $\{(G_i, {\bf c}_i)\}_{i=0}^{|V|-2}$ as described in \S\ref{sec:reduction} and \S\ref{sec:interchange}. Applying \eqref{eq:Eex}, \eqref{EIPineq2}, \eqref{eq:Eex}, and \eqref{eEX} in order, we obtain
\begin{align*}
\mathcal{E}^{\rm EX}_{(G,{\bf c}),\nu_\alpha}(f) & =\sum_{k=0}^{|V|} {|V|\choose k} \alpha^k (1-\alpha)^{|V|-k}\mathcal{E}^{\rm IP}_{(G,{\bf c})}(f_k\circ \pi_k) \\
\nonumber &\geq  \sum_{k=0}^{|V|} {|V|\choose k} \alpha^k (1-\alpha)^{|V|-k}\mathcal{E}^{\rm IP}_{(G_{|V|-2},{\bf c}_{|V|-2})}(f_k\circ \pi_k) \\ 
\nonumber &= \mathcal{E}^{\rm EX}_{(G_{|V|-2},{\bf c}_{|V|-2}),\nu_\alpha}(f) \\
&=  \frac{1}{2} c_{\rm eff}(x,y) \, \nu_{\alpha}[(\nabla_{xy} f)^2].
\end{align*}
Since $c_{\rm eff}(x,y) = [R_{\rm eff}(x,y)]^{-1}$, the theorem is proved.
\end{proof}

\subsection*{Acknowledgements}

I would like to thank Luke Rogers, Laurent Saloff-Coste, and especially Alexander Teplyaev for useful discussions concerning the contents of this paper. 
I am also grateful to Ryokichi Tanaka and the referees for valuable comments on an earlier version of this paper.

\vspace{10pt}

\end{document}